\documentclass[aip,jmp,author-year,reprint,eqsecnum]{revtex4-1}
\usepackage{natbib}
\usepackage{amsmath,amsfonts,amsthm,amssymb,amsbsy}
\usepackage{dsfont}
\usepackage{hyperref}
\def\e{\mathrm e}
\def\R{\mathbb R}
\def\C{\mathbb C}
\def\N{\mathbb N}
\newtheorem{thm}{Theorem}[section]
\newtheorem{teom}{Theorem}
\newtheorem{lem}{Lemma}
\newtheorem{cor}[lem]{Corollary}
\newtheorem{defi}{Definition}
\newtheorem{remark}[thm]{Remark}
\begin{document}
\title{Zero resonances for localised potentials}
\author{V. Georgiev}
\affiliation{Department of Mathematics; University of Pisa; 
Italy} \email{georgiev@dm.unipi.it}
\affiliation{Faculty of Science and Engineering; Waseda University;  
Tokyo 169-8555, Japan}
\affiliation{IMI--BAS, Acad. Georgi Bonchev Str., Block 8, 1113 Sofia, Bulgaria}
\author{J. Wirth}
\affiliation{Department of Mathematics; University of Stuttgart;  
Germany}
\email{jens.wirth@mathematik.uni-stuttgart.de}
\begin{abstract}
This paper considers Hamiltonians with localised potentials and gives a variational characterisation of resonant coupling parameters, which allow to provide estimates for the first resonant parameter and in turn also to provide bounds for resonant free regions. As application we provide a constructive approach to calculate the first resonant parameter for Yukawa type potentials in $\R^3$.
\end{abstract}
\maketitle
%
\section{Introduction}
In this introductory section we will provide necessary definitions and rigorous statements of the main results. Subsequent sections contain the proofs.

We consider the perturbation of the Laplace operator $\Delta$ in $\R^n,$ $n \geq 3$, by a potential $V$,
\begin{equation}
   H_\varkappa =  - \Delta - \varkappa V(x),
\end{equation}
where $\varkappa\in\R$ is a coupling parameter and the potential $V(x)$ is real, locally bounded on $\mathbb R^n\setminus\{0\}$ and satisfies the asymptotic properties
\begin{equation}\label{eq.H1}
    \lim_{x \to 0} |x|^2 V(x) = 0, \ \ \lim_{x \to \infty} |x|^2 V(x) = 0.
\end{equation}
These assumptions guarantee a form of compactness used later on and imply in particular that the essential spectrum of the operator satisfies $\sigma_{\rm ess}(H_\varkappa) = [0,\infty)$ for all values $\varkappa\in\R$.
We further assume that $V$ is somewhere positive.

Our first goal shall be to find the values for the coupling parameter $\varkappa$ so that $ H_\varkappa$ has zero resonances and/or zero eigenstates for this value of $\varkappa$.
 Before giving a more rigorous definition of zero resonances and zero resonant states, we can observe that a resonant state $u$ shall be a weak solution 
 of the   equation
\begin{equation}
\label{rests1} \Delta u + \varkappa V(x) u=0
\end{equation}
 belonging to the Dirichlet space $\dot H^1(\R^n).$
 
 It is well-known that in general the resonances on the real line may occur only at the origin \citep{GV07} and therefore it is meaningful to look for the properties of resonant states associated with zero resonance. The presence  of zero resonances can influence the decay rate of the solutions to the Schr\"odinger equation with potential or the wave equation with a potential \citep{ESc04, ESc06, GV03, DP05}.

As we mentioned above resonant states belong to the Dirichlet space $\dot H^1(\R^n)$ and   we recall that 
\begin{equation}
\dot H^1(\R^n) \subset L^{2^*}(\R^n)
\end{equation}
holds true with the Sobolev exponent $2^* = 2n/(n-2)$.  Thus   $ u \in L^2_{\rm loc}(\R^n)$ for $n \geq 3$.

\begin{defi} \label{def.r1}
If $u \in L^2_{\rm loc} (\R^n) $ is a solution to \eqref{rests1}
in distributional sense and
\begin{equation}
\int_{|x|>1}  |x|^{-2a} |u(x)|^2\mathrm d x <\infty
\end{equation} 
for some $a \in (1/2, 3/2)$, but $u$ is not in $ L^2(|x| \geq 1)$, then we shall say that the Hamiltonian $ H_\varkappa$ has zero resonances and $u$ is its zero resonant state or ground state. If $u$ is in $L^2(\R^n)$ we refer to $u$ as zero eigenstate and zero is an embedded eigenvalue.
\end{defi}

The Dirichlet space  $\dot H^1(\R^n)$ is the homogeneous Sobolev space of order one, i.e., the closure of $C_0^\infty(\R^n)$ with respect to the Dirichlet norm.
Using the Hardy inequality
\begin{equation}  \frac{(n-2)^2}{4} \int_{\R^n} |x|^{-2} |u(x)|^2 \mathrm d x \leq \|\nabla u\|_{L^2(\R^n)}^2 \end{equation}
  valid  for all $n \geq 3$, one can use  assumption
 \eqref{eq.H1} and conclude  that the operator
 \begin{equation}\label{eq:embComp}
  \dot{H}^1(\R^n) \ni u \mapsto  |V|^{1/2} u \in L^2(\R^n)
 \end{equation}
 can be approximated in the operator norm by the sequence of operators
 \begin{equation} \dot{H}^1(\R^n) \ni u \mapsto  \mathds{1}_{1/k \leq |x| \leq k}|V|^{1/2} u \in L^2(\R^n).
 \end{equation}
 Thus from the fact that for any $k \in \N$ this is a compact operator, we see that the operator in \eqref{eq:embComp}
 is  compact  from $\dot{H}^1(\R^n) $ into $L^2(\R^n)$.

Furthermore, the Hardy inequality implies that the variational problem
 \begin{equation}\label{eq.M1}
 \begin{split}
    J(V) &= \sup_{u \in \dot{H}^1(\R^n) \setminus \{0\}} \, \frac{ \int_{\R^n} V(x) |u(x)|^2 \mathrm d x}{\|\nabla u\|_{L^2}^2}
    \\&= \sup_{\|\nabla u\|_{L^2} = 1}  \int_{\R^n} V(x) |u(x)|^2 \mathrm d x
 \end{split}
  \end{equation}
characterises a finite quantity $J(V) \in \R$  and due to the compactness of \eqref{eq:embComp}  the supremum is attained at some element $u \in \dot{H}^1(\R^n)$. 

The compactness assumption is essential, in the case of the Hardy potential $V_H(x)= |x|^{-2}$ the supremum in \eqref{eq.M1} is given by  \citep{OK90, D99}
  \begin{equation}\label{eq.M2}
  \begin{split}
    J(V_H) &= \sup_{u \in \dot{H}^1(\R^n) \setminus \{0\}} \, \frac{ \int_{\R^n} |x|^{-2} |u(x)|^2 \mathrm d x}{\|\nabla u\|_{L^2}^2} \\&= \frac4{(n-2)^2}
  \end{split}
  \end{equation}
 but it is not attained in $\dot{H}^1(\R^n)$.

\begin{lem}\label{L1}
Assume $n\ge 3,$ the potential $V$ satisfies \eqref{eq.H1} and has a non-vanishing positive part.
 Then the supremum in the variational problem \eqref{eq.M1} is attained for some real-valued $u\in\dot H^1(\R^n)$.
\end{lem}

One natural question is to characterise open intervals free  of resonant parameters $\varkappa$. We are particularly interested in the neighbourhood of $\varkappa=0$.
The inequality $J(V)<\infty$ and the fact that any weak solution of $\Delta u + \varkappa V u = 0$ with  $(1+|x|)^{-a} u \in L^2(\R^n)$ is automatically in $\dot{H}^1(\R^n)$,
imply  the following.

\begin{lem}
If $n \geq 3,$ the potential satisfies the assumptions \eqref{eq.H1} and
$ \varkappa_*= (J(V))^{-1}$ with $J(V)$ defined in \eqref{eq.M1}, then for any
$\varkappa \in (0,\varkappa_*)$ the equation $H_\varkappa u=0$ has no solution in $L^2_{\rm loc}(\R^n)$ such that $(1+|x|)^{-a} u \in L^2(\R^n)$ with some $a \in (1/2,3/2)$.
\end{lem}

Our next result  gives a sufficient condition on $V$ such that the choice $\varkappa_*=1/J(V)$ implies $H_{\varkappa_*}$ has a resonance in $0$, i.e., the corresponding maximiser is a resonant state and in particular not in $L^2(\R^3)$. For this we assume that the potential $V$ is non-negative, i.e.,
\begin{equation}\label{eq.H2}
   V(x) \ge0 \qquad a.e.
\end{equation}
Then clearly from  the well-known property  \citep{Giu03}
\begin{equation}\|\  \nabla |u|\  \|_{L^2}  \leq \| \nabla u \|_{L^2} \end{equation}
we conclude that with each maximiser $u\in\dot H^1(\R^n)$ also its modulus $|u|$ is a maximiser and thus without loss of generality we may assume $u(x)\ge0$ a.e. in the following. Then $u$ is a solution to the corresponding Euler--Lagrange equation
\begin{equation}\label{eq.M3}
\begin{split}
 &  \Delta u + \varkappa_* V(x) u(x) = 0 , \ \ x \in \R^n,  \ x \neq 0, \\
  & u(x) > 0 , \ \ x \in \R^n.
\end{split}
\end{equation}

We formulate our first main theorem. The assumptions on the potential are slightly weaker than before. It applies in particular to maximizers of \eqref{eq.M1}.
\begin{teom} \label{MT1}
Suppose that $n \geq 3$, the potential belongs to the  Lorentz space  $V \in L^{p,n}(\R^n)$ with  $p = n/3$, is  non-negative and not identically vanishing.
Then any positive solution $u\in\dot H^1(\R^n)$
 to \eqref{eq.M3} is not in $L^2(\R^n)$ for $n\in\{3,4\}$ and
belongs to $\mathrm L^2(\R^n)$ for $n\ge5$.
\end{teom}
\begin{remark}
If $n\in\{3,4\}$ the assumption $V \in L^{p,n}(\R^n)$ is too strong for the result stated and can be  replaced by $V \in L^{p,n}_{\rm loc}(\R^n)$.
\end{remark}

The above result can be extended assuming
\begin{equation} V = V_+ - V_-, \ \ V_+(x) = \max (V(x),0), \end{equation}
and assuming $V_+$ satisfies the assumption of Theorem \ref{MT1} and $V_-$ decaying faster than $|x|^{-2}$ at infinity.

\begin{teom} \label{MT2} Suppose that $n \geq 3,$
 \begin{equation} V = V_+ - V_-, \ \ V_+(x) = \max (V(x),0), \end{equation}
 so that  $V_+(x) \in L^{p,n}(\R^n)$ with  $p = n/3$ and  $V_+(x) >0$ on a set of positive measure and  $V_-(x)$ satisfies the decay property
 \begin{equation} V_-(x) \lesssim (1+|x|)^{-b}, \ \ b>2.\end{equation}
 Then any maximiser of \eqref{eq.M1} which is solution to \eqref{eq.M3} is not in $L^2(\R^n)$ for $n\in\{3,4\}$
 and belongs to $L^2(\R^n)$ for $n\ge 5$.
\end{teom}

Our next step is to give a very precise estimate for the first resonant parameter $\varkappa_*(V_0) $ for the case, when $V_0$ is the Yukawa potential
\begin{equation}
   V_0(r) = \frac{\e^{-r}}r,\qquad r=|r|,
\end{equation}
and space dimension equals $3$. The estimate corresponds to bounds given in a recent paper \citep{EGSTW17}.
\begin{teom}\label{thm3}
Suppose $n=3$. Then we have the estimates
\begin{equation} 1.67626 < \varkappa_*\left( \frac{\e^{-r}}r\right) <1.68742.\end{equation}
\end{teom}

Finally, using a simple comparison argument based on the inequality
\begin{equation} 
\frac{\left( \int_{\R^n} V_1(x) |u(x)|^2 \mathrm d x\right)}{\|\nabla u\|_{L^2}^2} \leq \frac{\left( \int_{\R^n} V_2(x) |u(x)|^2 \mathrm d x\right)}{\|\nabla u\|_{L^2}^2}, \end{equation}
for all $V_1(x)\le V_2(x)$  we can assert that the implication
\begin{equation}\label{res2}
    V_1(x) \leq V_2(x) \ \ \Longrightarrow \ \ \varkappa_*(V_1) \geq \varkappa_*(V_2)
\end{equation}
holds true. In this way we arrive at the following.

\begin{teom}
Supppose $n=3$. If the potential satisfies the assumptions \eqref{eq.H1}and
\begin{equation} V(x) \leq C_0 V_0(x),\end{equation} where $V_0$ is the Yukawa potential, then the first resonant parameter $\varkappa_*(V)$ satisfies the estimate
\begin{equation} \varkappa_*(V) \geq \frac{\varkappa_*(V_0)}{C_0} > \frac{1.67626}{C_0}.\end{equation}
\end{teom}

Another scenario of interaction of two potential is associated with the Hamiltonian
 \begin{equation}\label{eq.M4}
   H_{\varkappa_1, \varkappa_2} =   -\Delta - \varkappa_1 V_1(x) - \varkappa_2 V_2(x),
\end{equation}
 where $V_1$ is a Hardy type potential,  $0 \le V_1(x) \le |x|^{-2}$, the coupling constant $\varkappa_1 \in (0, (n-2)^2/4)$ is chosen so that
 $-\Delta - \varkappa_1 V_1$ is positive and one can introduce zero resonant  state and zero resonance as in  Definition \ref{def.r1}.
\begin{defi} \label{def.r2}
If $u \in L^2_{\rm loc} (\R^n) $ is a solution to \begin{equation}  \Delta u + \varkappa_1 V_1(x) u + \varkappa_2 V_2(x) u  =0\end{equation}
in distribution sense and $(1+|x|)^{-a} u \in L^2(\R^n)$ for some $a \in (1/2, 3/2)$, but $u$ is not in $ L^2(\R^n),$ then we shall say that the Hamiltonian $ H_{\varkappa_1,\varkappa_2}$ has
 zero resonances and $u$ is its zero resonant state or ground state.
\end{defi}

To obtain the first resonant parameter $\varkappa_2,$ we consider
the variational problem
 \begin{equation}\label{eq.M5}
 \begin{split}
    &J_{\varkappa_1}(V_1, V_2) \\
    &= \sup_{u \in \dot{H}^1(\R^n) \setminus \{0\}} \, \frac{\left( \int_{\R^n} V_2(x) |u(x)|^2 \mathrm d x\right)}{\|\nabla u\|_{L^2}^2- \varkappa_1 \int_{\R^n} V_1(x) |u(x)|^2 \mathrm d x},
    \end{split}
  \end{equation}
 where $V_2$ is a potential satisfying the compactness assumptions \eqref{eq.H1}.
 In this way one obtains again finiteness of $J_{\varkappa_1}(V_1,V_2)$ and verify that
  the supremum is attained at some $u \in \dot{H}^1(\R^n).$

Our next result gives a sufficient condition on $V_2$ such that the choice of the  parameter $\varkappa_2=1/J_{\varkappa_1}(V_1,V_2)$ with $J_{\varkappa_1}(V_1,V_2)$  defined in \eqref{eq.M5} guarantees that we have a resonance in $0$ and the corresponding ground state is not in $L^2(\R^3).$ We again assume that $V_2(x)\ge0$ almost everywhere. Then without loss of generality the maximiser $u\in\dot H^1(\R^n)$ of \eqref{eq.M5} is non-negative  and thus a solution to the corresponding Euler-Lagrange equation
\begin{equation}\label{eq.M6}
\begin{split}
 &  \Delta u + \varkappa_1 V_1(x) u(x) + \varkappa_2 V_2(x) u(x) = 0 , \ \ x \in \R^n, \\
  & \nonumber u(x) > 0 , \ \ x \in \R^n,  \ x \neq 0.
\end{split}
\end{equation}

\begin{teom} \label{MT3} Suppose that $n \geq 3,$ $V_1(x)=|x|^{-2}$,  $V_2(x)$ satisfies the assumptions \eqref{eq.H1},
\begin{equation} V_2 = V_{2,+} - V_{2,-}, \ \ V_{2,+}(x) = \max (V_2(x),0), \end{equation}
 so that
 $V_{2,+} \in L^{p,\infty}(\R^n)$, $p=n/3$ and $V_-(x)$ satisfies the decay property
 \begin{equation} V_-(x) \lesssim (1+|x|)^{-b}, \ \ b>2.\end{equation}
  Then the maximizer of \eqref{eq.M5} which is solution to \eqref{eq.M6} is not in $L^2(\R^n)$ for $n\in\{3,4\}$ and belongs to $L^2(\R^n)$ for $n\ge 5$.
\end{teom}

\section{Proofs of the variational formulae}

In this section we investigate the variational problem \eqref{eq.M1} and prove some of our main theorems. For convenience we start with the proof of Lemma~\ref{L1}. As both numerator and denominator of the variational quotient are homogeneous, we can replace the supremum over all functions from $\dot H^1(\R^n)\setminus\{0\}$ by a supremum over the unit sphere or the unit ball in $\dot H^1(\R^n)$.

\begin{proof}[Proof of Lemma~\ref{L1}]  The quantity \eqref{eq.M1} is a well-defined positive number if the potential has a non-vanishing positive part. If $u_j\in\dot H^1(\R^n)$ is a real-valued maximizing sequence for \eqref{eq.M1} with $\|\nabla u_j\|_2=1$, then we have to show that this sequence can not converge weakly in $\dot{H}^1$ to zero. By weak compactness of the unit ball in $\dot H^1(\R^n)$ there exists a convergent subsequence (also denoted by $u_j$) and a $u_*$ with
$u_j \rightharpoonup u_*$ in $\dot H^1(\R^n)$. By compactness of the embedding \eqref{eq:embComp}
\begin{equation}
    \int_{\R^n} |V(x)| \,| u_j(x) - u_*(x)|^2 \mathrm d x \to 0
\end{equation}
and hence by Cauchy-Schwarz
\begin{equation}
  \begin{split}
&  \int_{\R^n} |V(x)| \, \big(u_j(x)-u_*(x)\big) u_*(x) \mathrm d x \to 0,\\
 &   \int_{\R^n} |V(x)|\, u_j(x) \big(u_j(x)-u_*(x)\big)\mathrm d x\to 0
 \end{split}
\end{equation}
such that
\begin{equation*}
\begin{split}
& \left| \int_{\R^n} V(x) |u_j(x)|^2 \mathrm d x - \int_{\R^n} V(x) |u_*(x)|^2 \mathrm d x \right|
\\&  \le \int_{\R^n} |V(x)| \, \left|  |u_j(x)|^2 - |u_*(x)|^2\right| \mathrm d x \\
& =
\int_{\R^n} |V(x)| \, |u_j(x)-u_*(x)|^2 \mathrm d x \\
&+ 2 \int_{\R^n} |V(x)|\,  \big(u_j(x)u_*(x)  - \min( |u_*(x)|^2, |u_j(x)|^2 \big)\mathrm d x
\end{split}
\end{equation*}
tends to zero. In particular this implies
\begin{equation}
  0\ne J(V) = \int_{\R^n} V(x) |u_*(x)|^2 \mathrm d x 
\end{equation}
 so  $0 < \|\nabla u_*\|_2 \leq \lim_{j \to \infty}\|\nabla u_j\|_2 = 1. $
This observation shows that $u_*$ is one maximizer and the lemma is proven.
\end{proof}

\begin{proof}[Proof of Theorem \ref{MT1}]
We distinguish two parts, first we give an upper bound on the solution based on convolution inequalities. This is valid in space dimensions $n\ge 5$ and gives square integrability of solutions. In a second part we give a lower bound for space dimensions $n\in\{3,4\}$ contradicting square integrability.

{\it Part 1.}  We use H\"older inequality combined with the Young-O'Neil inequality of convolutions \citep{ON63} to estimate $u$ from the a-priori knowledge $u\in \dot H^1(\R^n)$. First, by Sobolev inequality
$u\in L^{2^*}(\R^n)= L^{2^*,2^*}(\R^n)$ for the Sobolev exponent $2^*=2n / (n-2)$.
Therefore by H\"older inequality the product $Vu$ belongs to the Lorentz space
\begin{equation}
   Vu \in  L^{q,2}(\R^n)
\end{equation}
with $1/q = 1/p + 1/2^*$ and using $1/2 = 1/2^* + 1/n$. Note that $q=2n/(n+4) > 1$ for $n \geq 5$.

 Thus the convolution
\begin{equation}
w(x) =\int_{\R^n} \frac{ V(y) u(y) }{|x-y|^{n-2}}\mathrm d y
\end{equation}
with the fundamental solution of the Laplacian (which  belongs to $L^{n/(n-2),\infty}(\R^n)$) yields a function
\begin{equation}
   w \in L^{2,2}(\R^n),\qquad  1+ \frac12 = \frac1q + \frac{n-2}{n}.
\end{equation}
As $w=u$ for the maximiser the statement follows.

{\it Part 2.} Let now $n\in\{3,4\}$ and $u\in\dot H^1(\R^n)$ be a non-negative solution to \eqref{eq.M2}.
By Sobolev inequality we know that
\begin{equation}\label{eq.FR0}
   u \in  L^{2^*}(\R^n),
\end{equation}
with $2^*=2n/(n-2)>2$. We use the H\"older inequality to conclude from $1/2^* + 1/q <1$ the local integrability of the product
\begin{equation}\label{eq.FR1}
   Vu \in  L^{1}_{\rm loc}(\R^n)
\end{equation}
such that $u>0$ implies
\begin{equation}
    \int_{|x|\le R} V(x) u(x) \mathrm d x \geq 2 C_0 >0
\end{equation}
for sufficiently large $R$. Using the equation
\begin{equation}
-\Delta u = \varkappa_* V u ,
\end{equation}
we take $|x| \ge 2 R$ and  find
\begin{equation}
\begin{split}
u(x) &=  c\varkappa_* \int_{\R^n} \frac{ V(y) u(y) }{|x-y|^{n-2}}\mathrm d y \\&\geq c\varkappa_*  \int_{ |y| \leq R}\frac{ V(y) u(y)}{|x-y|^{n-2}} \mathrm d y \geq \frac{C_1}{|x|^{n-2}}.
\end{split}
\end{equation}
Since the function $|x|^{2-n}$ is not in $L^2(\{ |x|>2 R \})$ for $n \in\{3,4\}$ we conclude that $u\not\in\mathrm L^2(\R^n)$.
\end{proof}

\begin{proof}[Proof of Theorem \ref{MT2}]
As the first part of previous proof did not require a sign condition on $V$, it suffices to provide the lower bounds for
$n\in \{3,4\}$. Assume that  $u \in  \dot H^1(\R^n)$ is positive. Then  \eqref{eq.FR0}  and \eqref{eq.FR1} with $V_+$ instead of $V$ are valid, so we obtain again
\begin{equation}\label{eq.FR2}
    \int_{|x| \leq R} V_+(x) u(x) \mathrm d x \geq C_0 >0
\end{equation}
for sufficiently large $R$.
The equation
\begin{equation} -\Delta u = \varkappa_* V u ,\end{equation}
can be rewritten as
\begin{equation}\label{eq.FR3}
   -\Delta u + \varkappa_* V_- u = \varkappa_* V_+ u.
\end{equation}
Setting $A = -\Delta + \varkappa_* V_-,$ we can represent the resolvent $(\lambda +A)^{-1}$ for $\lambda >0$ as Laplace transform of the associated semigroup $\e^{-At}$
\begin{equation} (\lambda +A)^{-1} = \int_0^\infty \e^{-\lambda t} \e^{-At} \mathrm d t,\end{equation}
and using the assumptions on $V_-$ we can apply Theorem~1.2 from Zhang's paper \citep{Z00} and see that the kernel $\e^{-tA}(x,y)$ of the heat semigroup associated with $A$ satisfies the lower Gaussian bound
\begin{equation} \e^{-tA}(x,y) \gtrsim \frac{\e^{-c|x-y|^2/t}}{t^{n/2}}. \end{equation}
 Hence, we have
 \begin{equation} (\lambda +A)^{-1}(x,y) \gtrsim  \int_0^\infty \e^{-\lambda t} \e^{-c|x-y|^2/t}  \frac{\mathrm dt}{t^{n/2}} \end{equation}
 and using a change of variables $ t \to s= |x-y|^2/t, $ we find
 \begin{equation}
 \begin{split}
&  \int_0^\infty \e^{-\lambda t} \e^{-c|x-y|^2/t}  \frac{\mathrm d t}{t^{n/2}} \\&= |x-y|^{2-n}
 \int_0^\infty \e^{-\lambda |x-y|^2/s} \e^{-cs}  s^{(n-4)/2}\mathrm d s
 \end{split}
 \end{equation}
 and taking the limit $ \lambda \searrow 0,$ we see that
 \begin{equation}(-\Delta+\varkappa_* V_-)^{-1}(x,y) \geq c |x-y|^{2-n}\end{equation}
 provided the space dimension satisfies $n\ge3$.
 Turning to the equation \eqref{eq.FR3}, we use the lower-bound for the kernel of $(-\Delta+\varkappa_* V_-)^{-1}$ and deduce for $|x|>2R$
\begin{equation}
\begin{split}
 u(x) &\gtrsim  \int_{\R^n} |x-y|^{-(n-2)} V_+(y) u(y) \mathrm d y 
\\&\gtrsim   \int_{ |y| \leq R} |x|^{-(n-2)} V_+(y) u(y) \mathrm d y \gtrsim |x|^{2-n}.
\end{split}
\end{equation}
This lower bound implies that $u$ can not be an element of $\mathrm L^2(\R^n)$ and this completes the proof.
\end{proof}

\begin{proof}[Proof of Theorem \ref{MT3}]
The proof is practically a repetition of the proof of Theorem \ref{MT2}, since the equation
\begin{equation}
\begin{split}
 (-\Delta &+ \varkappa_2 V_{2,-}) u \\
 &= \varkappa_1 |x|^{-2} u(x) + \varkappa_2 V_{2,+}(x) u(x) \\
 &\geq \varkappa_1 |x|^{-2} u(x)
 \end{split}
 \end{equation} combined with the estimate
\begin{equation} (-\Delta + \varkappa_2 V_{2,-})^{-1}(x,y) \gtrsim |x-y|^{2-n}  \end{equation}
is sufficient to obtain a contradiction, assuming that the maximiser of \eqref{eq.M5} is in $L^2(\R^n)$ for $n\in\{3,4\}$.
\end{proof}

\section{Radial potentials and  resonant states}\label{sec2}
We assume now that the  potential $V(x)$ is non-negative, radial and decreasing. Let further $u$ be a  non-negative maximiser of \eqref{eq.M1} and denote by $u_*$ the Schwarz symmetrisation (spherically symmetric decreasing rearrangement) of $u$,
\begin{equation}
  u_*(x) = \sup \left\{ w\ge0 \;:\;   \int_{u(y)\ge w} \mathrm d y \ge  \int_{|y|\le |x|} \mathrm d y  \right\}  .
\end{equation}
Then  the well-known rearrangement inequalities
\begin{equation}
\begin{split}
   \int_{\R^n} V(|x|) u(x)^2 \mathrm d x 
   \le \int_{\R^n} V(|x|) u_*(x)^2 \mathrm d x
   \end{split}
\end{equation}
and
\begin{equation}
   \int_{\R^n} |\nabla u_*(x)|^2\mathrm d x \le  \int_{\R^n} |\nabla u(x)|^2\mathrm d x
\end{equation}
hold true \citep{PS, Kawohl}. 
Hence the variational quotient in \eqref{eq.M1} is increasing when applying Schwarz symmetrisation 
and as $u$ was assumed to be a maximiser it stays constant. Hence a radially symmetric maximiser, i.e., radial a resonant or ground state, exists. 

We can  compute such radial states as solution of the ordinary
differential equation
\begin{equation}\label{eq:ode1}
   (r\partial_r)^2 u + (n-2) r\partial_r u + \varkappa r^2 V(r) u = 0.
\end{equation}
In this section we will discuss how to use an asymptotic integration argument for small and large values of $r$ in order to characterise the first resonant value $\varkappa_*(V) = 1/J(V)$ as first zero of a suitable holomorphic function and provide estimates for it in this way.

Our aim is to apply this to the particular case of the Yukawa potential $V(r) = \e^{-r}/r$, but calculations itself are based on the more general assumption
\begin{equation}\label{eq.H2}
   \int_0^\infty r|V(r)| \mathrm d r < \infty.
\end{equation}
They imply in particular that the potential term in the above Fuchs type differential equation is a small perturbation and the asymptotic behaviour of its solutions is described by solutions to the indicial equation $\rho^2 + (n-2)\rho = 0$, i.e. by $\rho=0$ and $\rho = 2-n$. We assume again $n\ge 3$. Hence, by applying an adapted version of Levinson's theorem  \citep{Eastham, NW15} we obtain a fundamental system of solutions of the form
\begin{equation}
\begin{split}
    u_{\infty,1}(r) &= 1  + \mathbf o(1),\\
     u_{\infty,2}(r) &=  r^{2-n} + \mathbf o(r^{2-n}),\qquad r\to\infty
\end{split}
\end{equation}
near infinity and a fundamental system of solutions of the form
\begin{equation}
\begin{split}
  &  u_{0,1} (r) =1  + \mathbf o(1),\\
  &  u_{0,2}(r) =  r^{2-n}+ \mathbf o(r^{2-n}), \qquad r\to0
\end{split}
\end{equation}
near the origin. The asymptotic behaviour for $r\partial_r u$ is governed by the same terms, i.e. the limits
\begin{equation}
\begin{split}
   \lim_{r\to0} r\partial_r   u_{0,1} (r),&\qquad   \lim_{r\to0} r^{n-1} \partial_r   u_{0,2} (r),\\
   \lim_{r\to\infty} r\partial_r   u_{\infty,1} (r), &\qquad   \lim_{r\to\infty} r^{n-1}\partial_r   u_{\infty,2} (r)
\end{split}
\end{equation}
exist.

In order to single out the ground state we have to find a solution
behaving like $r^{2-n}$ at infinity and being bounded in the origin, otherwise  it is not in $\dot H^1(\R^n)\subset \mathrm L^2_{\rm loc}(\R^n)$. Thus, we are in a radial resonant state
if and only if $\mathbf u_{\rm int}(\cdot;\varkappa)=u_{0,1}$ and $\mathbf u_{\rm ext}(\cdot;\varkappa)=u_{\infty,2}$ are linearly dependent.

%
We use an asymptotic integration argument to construct
these particular solutions.

\begin{lem}\label{lem31}
For each $\varkappa\in\C$ there exists a uniquely determined solution $\mathbf u_{\rm int}(\cdot;\varkappa)$ to \eqref{eq:ode1}
satisfying the asymptotic constraint
\begin{equation}
 \mathbf u_{\rm int}(r;\varkappa) = 1 + \mathbf o(1),\qquad r\to0.
\end{equation}
This solution is entire in $\varkappa$.
\end{lem}
\begin{proof}
Making use of the ansatz $ \mathbf u_{\rm int}(r;\varkappa) = 1+ w(r;\varkappa)$ with $w(0;\varkappa)=0$, we obtain from
\begin{equation}
   \partial_r \bigg( r^{n-1} \partial_r w(r)\bigg)+ \varkappa r^{n-1} V(r) \big( 1+w(r) \big) = 0
\end{equation}
by integration
\begin{equation}
   r^{n-1} \partial_r w(r)  +  \varkappa \int_0^r s^{n-1} V(s) \big(1+w(s)\big) \mathrm d s = 0
\end{equation}
and therefore integrating a second time and changing order of integration
\begin{equation}
   w(r) + \varkappa \int_0^r s V(s) \big(1+w(s)\big)  \frac{r^{n-2}-s^{n-2}}{(n-2)r^{n-2}} \mathrm d s = 0.
\end{equation}
This is a Volterra integral equation with unique solution $w$ expressible as Neumann series
for all values of $r>0$ and $\varkappa\in\mathbb C$. As the Neumann series is a power series in $\varkappa$, the solution and hence $\mathbf u_{\rm int}(r,\varkappa)$ is entire in $\varkappa\in\C$.
\end{proof}

\begin{lem}\label{lem32}
For each $\varkappa\in\C$ there exists a uniquely determined solution $\mathbf u_{\rm ext}(\cdot;\varkappa)$ to \eqref{eq:ode1}
satisfying the asymptotic constraint
\begin{equation}
 \mathbf u_{\rm ext}(r;\varkappa) = {r^{2-n}} + \mathbf o(r^{2-n}),\qquad r\to\infty.
\end{equation}
This solution is entire in $\varkappa$.
\end{lem}

\begin{proof}
Making the ansatz $u(r)= r^{2-n} (1 +w(r))$ with decaying $w(r)$, the differential equation
\begin{equation}
   \partial_r \bigg( r^{n-1} \partial_r \big( r^{2-n} w(r)\big)\bigg)+ \varkappa r V(r)  \big( 1+w(r) \big) = 0
\end{equation}
rewrites again as
\begin{equation}
   r^{n-1} \partial_r \big( r^{2-n} w(r)\big)
     -  \varkappa \int_r^\infty s V(s) \big(1+w(s)\big) \mathrm d s = 0
\end{equation}
and hence after a second integration as Volterra integral equation
\begin{equation}
  w(r) + \varkappa \int_r^\infty s V(s) \big(1+w(s)\big)  \frac{s^{n-2}-r^{n-2}}{(n-2)s^{n-2}} \mathrm d s = 0.
\end{equation}
after integrating twice. Again its solution is given by a Neumann series which is convergent for all $r>0$ and all parameters $\varkappa\in\C$.
\end{proof}

We denote by
\begin{equation}
\mathcal W_V(\varkappa) =
\begin{vmatrix}  
  \mathbf u_{\rm ext}(1;\varkappa) & \mathbf u_{\rm int}(1;\varkappa)  \\
 \mathbf u_{\rm ext}'(1;\varkappa) &  \mathbf u'_{\rm int}(1;\varkappa)
\end{vmatrix}
\end{equation}
the Wronskian of the two particular solutions just constructed. By definition, we are in a radial resonant state, if both are linearly dependent (as both are the only solutions in the right spaces) and thus if $\mathcal W_V(\varkappa)=0$. The function $\mathcal W_V$ is entire and $\mathcal W_V(0)=n-2\ne0$.

\begin{lem}
Assume $V$ is radial, non-negative and satisfies \eqref{eq.H1} together with \eqref{eq.H2}.
Then the first resonant value $\varkappa_*(V)$ is the smallest real zero of the Wronskian,
\begin{equation}
  \varkappa_*(V) = \min\{\varkappa\in\R \;:\; \mathcal W_V(\varkappa)=0\}.
\end{equation}
\end{lem}
\begin{proof}
By construction $H_\varkappa$ has a zero resonance or zero eigenstate for all $\varkappa$ with $\mathcal W_V(\varkappa)=0$. If $V\ge0$ is non-negative it thus follows that $\mathcal W_V(\varkappa)>0$ for $\varkappa<0$ and the statement follows.
\end{proof}

\section{Estimates for the first Yukawa resonance in $\mathbb R^3$}
In this section we prove Theorem~\ref{thm3} and provide explicit bounds for the first resonance in dimension $n=3$ for the Yukawa potential 
\begin{equation}
   V(r) = \frac{\e^{-r}}r.
\end{equation}
In order to prove them we use the asymptotic integration argument introduced in the previous section for large arguments, but make use of the analyticity of $rV(r)$ and the resulting analyticity of the resonance functions for small values of $r$.

To construct $\mathbf u_{\rm int}(r;\varkappa)$ we use the ansatz
\begin{equation}\label{eq:gb-ans}
  \mathbf u_{\rm int}(r;\varkappa) = 1 + \sum_{k=1}^\infty (-1)^k \alpha_k(\varkappa) r^k
\end{equation}
as power series with still to be determined coefficients $\alpha_k$. Indeed, plugging this ansatz into the differential equation \eqref{eq:ode1} yields (with $\alpha_0(\varkappa)=1$)
\begin{equation}
\begin{split}
   \sum_{k=2}^\infty &(-1)^k \alpha_k(\varkappa) k(k-1) r^{k-1} 
   + 2 \sum_{k=1}^\infty (-1)^k \alpha_k(\varkappa) k r^{k-1}
   \\ &+ \varkappa \left(\sum_{k=0}^\infty \frac{(-1)^k}{k!} r^k \right)\left( \sum_{k=0}^\infty  (-1)^k \alpha_k(\varkappa) r^k\right) = 0
\end{split}
\end{equation}
such that comparing coefficients yields the equations 
\begin{equation}
2\alpha_1(\varkappa)=\varkappa\alpha_0(\varkappa)  =  \varkappa
\end{equation}
 and
\begin{equation}\label{eq:rec1}
  \alpha_{k+1} (\varkappa) = \frac{\varkappa}{ (k+1) (k+2)} \sum_{\ell=0}^k \frac{\alpha_{k-\ell}(\varkappa)}{\ell!}.
\end{equation}
The last formula explains the sign convention used in ansatz \eqref{eq:gb-ans}, all appearing coefficients
$\alpha_k(\varkappa)$ are positive for positive $\varkappa$.
The first few terms are given below
\begin{equation}\label{eq:alf}
\begin{split}
    \alpha_1(\varkappa) &= \frac{\varkappa}{2},\\
     \alpha_2(\varkappa) &= \frac{\varkappa(\varkappa+2)}{12},\\
     \alpha_3(\varkappa) &= \frac{\varkappa(6+8\varkappa+\varkappa^2)}{144},\\
     \alpha_4(\varkappa) & = \frac{\varkappa(24+66\varkappa+20\varkappa^2+\varkappa^3)}{2880}, \\
     \alpha_5(\varkappa) &= \frac{\varkappa(120+624 \varkappa+ 346 \varkappa^2 + 40\varkappa^3 + \varkappa^4)}{86400} ,\\
     \alpha_6(\varkappa) &= \frac{\varkappa(720+6840\varkappa+6204 \varkappa^2 + 1246 \varkappa^3 + 70\varkappa^4 + \varkappa^5)}{3628800}    .
\end{split}
\end{equation}
If we can show that for $k$ sufficiently large (depending on $\varkappa$), the coefficients $\alpha_k(\varkappa)$
and $k\alpha_k(\varkappa)$ are monotonically decreasing in $k$, then the following error bounds
\begin{equation}
\begin{split}
     \sum_{k=0}^{2K-1} (-1)^k\alpha_k(\varkappa) <  \mathbf u_{\rm int}(1;\varkappa) <     \sum_{k=0}^{2K} (-1)^k\alpha_k(\varkappa),\\
      \sum_{k=1}^{2K-1} (-1)^k k \alpha_k(\varkappa) <  \mathbf u_{\rm int}'(1;\varkappa) <     \sum_{k=0}^{2K} (-1)^k k \alpha_k(\varkappa)
\end{split}
\end{equation}
follow directly from Leibniz criterium. To establish this monotonicity we make use of the following lemma.

\begin{lem}\label{lem:dec}
Assume for $k_0$ depending on $\varkappa$
\begin{equation}\label{eq:ind-hyp}
  \alpha_{k_0+1}(\varkappa) < \min_{k\le k_0}\alpha_k(\varkappa).
\end{equation}
Then for all $k\ge k_0$ the estimates
 \begin{equation} \label{eq:dec}
      \alpha_{k+1}(\varkappa)< \alpha_k(\varkappa)
  \end{equation}
  and for $k\ge \max\{k_0, 3\}$
  \begin{equation}\label{eq:dec1}
      (k+1)\alpha_{k+1}(\varkappa)< k\alpha_k(\varkappa)
  \end{equation}
  are valid.
\end{lem}
\begin{proof}
We prove  \eqref{eq:dec} by induction. The hypothesis \eqref{eq:ind-hyp} implies $\alpha_{k_0+1}(\varkappa) < \alpha_{k_0}(\varkappa)$. Let now $k> k_0$ and assume
\eqref{eq:dec} for all intermediate values. In particular  \eqref{eq:ind-hyp}  is true with $k_0$ replaced by $k-1$. Then we obtain from
the relation \eqref{eq:rec1}
\begin{equation}
\begin{split}
  \alpha_{k+1} (\varkappa)&- \alpha_k(\varkappa) \\&= 
   \frac{\varkappa}{ (k+1) (k+2)} \sum_{\ell=0}^k \frac{\alpha_{k-\ell}(\varkappa)}{\ell!} \\
   &\quad-
\frac{\varkappa}{ k (k+1)} \sum_{\ell=1}^{k} \frac{\alpha_{k-\ell}(\varkappa)}{(\ell-1)!} \\
&= \frac{\varkappa}{ (k+1) (k+2)} \alpha_k(\varkappa) - R_k(\varkappa),
\end{split}
\end{equation}
where
\begin{equation}
\begin{split}
    R_k(\varkappa) &=
\frac{\varkappa}{ k (k+1)} \sum_{\ell=1}^{k} \frac{\alpha_{k-\ell}(\varkappa)}{(\ell-1)!}\\
&\quad- \frac{\varkappa}{ (k+1) (k+2)} \sum_{\ell=1}^k \frac{\alpha_{k-\ell}(\varkappa)}{\ell!}.
\end{split}
\end{equation}
We estimate the term $R_k(\varkappa)$ from below based on the induction hypothesis as follows
\begin{equation}
    R_k(\varkappa) \geq
\frac{\varkappa \alpha_k(\varkappa)}{ k (k+1)(k+2)} \sum_{\ell=1}^{k} \frac{(\ell-1)k+2\ell}{\ell!}.
\end{equation}
To show that $\alpha_{k+1}(\varkappa) < \alpha_k(\varkappa)$ it is sufficient to show the inequality
\begin{equation}
\frac{\varkappa \alpha_k(\varkappa)}{ k (k+1)(k+2)} \sum_{\ell=1}^{k} \frac{(\ell-1)k+2\ell}{\ell!} \geq
\frac{\varkappa}{ (k+1) (k+2)} \alpha_k(\varkappa)
\end{equation}
or equivalently
\begin{equation}\label{eq:fac1}
   \sum_{\ell=1}^{k} \frac{(\ell-1)k+2\ell}{\ell!} \geq
k.
\end{equation}
Now we can use the relation
\begin{equation}\label{eq.rel1}
    \sum_{\ell=1}^{k} \frac{(\ell-1)}{\ell!} =\sum_{\ell=1}^{k} \frac{1}{(\ell-1)!}- \sum_{\ell=1}^{k} \frac{1}{\ell!} = 1 - \frac{1}{k!}
\end{equation}
and see that \eqref{eq:fac1} can be rewritten in the form
\begin{equation}
k - \frac{1}{(k-1)!} + 2 \sum_{\ell=1}^{k} \frac{1}{(\ell-1)!} \geq k.
\end{equation}
Since this inequality easily follow from the obvious one
\begin{equation}
  \sum_{\ell=1}^{k} \frac{1}{(\ell-1)!} \geq \frac{1}{(k-1)!},
\end{equation}
we can conclude that \eqref{eq:fac1} is established and so  the inequality \eqref{eq:dec} is proved.

The next step is the proof of \eqref{eq:dec1}. Using the recurrence relation
\eqref{eq:rec1} we obtain
\begin{equation}
\begin{split}
  (k+1)\alpha_{k+1} (\varkappa)&- k\alpha_k(\varkappa) \\
  &=  \frac{\varkappa}{ (k+2)} \sum_{\ell=0}^k \frac{\alpha_{k-\ell}(\varkappa)}{\ell!} \\&\quad-
\frac{\varkappa}{ (k+1)} \sum_{\ell=1}^{k} \frac{\alpha_{k-\ell}(\varkappa)}{(\ell-1)!} \\
&= \frac{\varkappa}{(k+2)} \alpha_k(\varkappa) - \tilde{R}_k(\varkappa),
\end{split}
\end{equation}
where the remainder is given by
\begin{equation}
    \tilde{R}_k(\varkappa) =
\frac{\varkappa}{ (k+1)} \sum_{\ell=1}^{k} \frac{\alpha_{k-\ell}(\varkappa)}{(\ell-1)!}- \frac{\varkappa}{ (k+2)} \sum_{\ell=1}^k \frac{\alpha_{k-\ell}(\varkappa)}{\ell!}.
\end{equation}
Similar to the above reasoning the remainder can be estimated for $k>k_0$ from \eqref{eq:ind-hyp} combined with \eqref{eq:dec}.
This yields
\begin{equation}
    \tilde{R}_k(\varkappa) \geq
\frac{\varkappa \alpha_k(\varkappa)}{ (k+1)(k+2)} \sum_{\ell=1}^{k} \frac{(\ell-1)k+2\ell-1}{\ell!}.
\end{equation}
Hence it is sufficient to show the inequality
\begin{equation}
 \frac{\varkappa \alpha_k(\varkappa)}{(k+1)(k+2)} \sum_{\ell=1}^{k} \frac{(\ell-1)k+2\ell-1}{\ell!} \geq
\frac{\varkappa}{ (k+2)} \alpha_k(\varkappa)
\end{equation}
or equivalently
\begin{equation}\label{eq:fac1a}
   \sum_{\ell=1}^{k} \frac{(\ell-1)k+2\ell-1}{\ell!} \geq
k+1.
\end{equation}
We can use again the relation \eqref{eq.rel1}
and see that \eqref{eq:fac1a} can be rewritten in the form
\begin{equation}
 k - \frac{1}{(k-1)!} +  \sum_{\ell=1}^{k} \frac{2\ell-1}{\ell!} \geq k+1.
 \end{equation}
If we assume that $k\ge 3$ in addition to $k<k_0$ it follows from the obvious estimate
\begin{equation}
\sum_{\ell=1}^{k} \frac{2\ell-1}{\ell!} \geq 1+\frac{1}{(k-1)!}  , \ \ k \geq 3,
\end{equation}
we can conclude that \eqref{eq:fac1a} is established and so  the lemma is proven.
\end{proof}

As the function $\mathbf u_{\rm int}(r;\varkappa)$ is entire in $r$  we know that for any fixed $\varkappa\ge0$
\begin{equation}
  \lim_{k\to\infty} \sqrt[k]{\alpha_k(\varkappa)} = 0.
\end{equation}
In particular, it follows $\alpha_k(\varkappa)\to0$ as $k\to\infty$ and thus the assumption \eqref{eq:ind-hyp} is satisfied at some point $k_0$. For small values of $\varkappa$ estimates for the number $k_0$ can be obtained explicitly.

\begin{cor}
 Assume $\varkappa\in(0,\sqrt{13}-1)$. Then \eqref{eq:dec}  and \eqref{eq:dec1} are both valid for all $k\ge 3$.
\end{cor}
\begin{proof}
$\alpha_2(\varkappa)\le 1$ is equivalent to $(\varkappa+2)/6<1$, while $\alpha_2(\varkappa)<\alpha_1(\varkappa)$ means $\varkappa(\varkappa+2)/12<1$. Both inequalities are valid for the given range of $\varkappa$.
\end{proof}

To construct $\mathbf u_{\rm ext}(r;\varkappa)$ we use asymptotic integration, but give the appearing integrals explicitly. This yields
\begin{equation}\label{eq:hb-ans}
  \mathbf u_{\rm ext}(r;\varkappa) = \frac1r \left( 1+ \sum_{k=1}^\infty (-1)^k \varkappa^k \omega_{k}(r) \right),
\end{equation}
where $\omega_k(r)$ is defined recursively by
\begin{equation}\label{eq:hb-k-1}
  r\omega_{1}''(r) =\e^{-r},\qquad \omega_{1}(r)\to0
\end{equation}
and
\begin{equation}
 r \omega_{k+1}''(r) = \e^{-r}\omega_{k}(r) ,\qquad \omega_{k+1}(r)\to0,
\end{equation}
limits taken as $r\to\infty$.
It turns out that all appearing functions are positive, $\omega_{k}(r)>0$. Again this explains the sign convention used in ansatz \eqref{eq:hb-ans}. Furthermore, the functions $\omega_k(r)$ are independent of $\varkappa$.

\begin{lem}
For any $r\ge 1$ the estimates
\begin{equation}
   0 < \omega_{k}(r) \le \frac{\e^{-kr}}{(k!)^2}
\end{equation}
and
\begin{equation}
\frac{-\e^{-kr}}{k!(k-1)!} <     \omega_{k}' (r) <0
\end{equation}
hold true.
\end{lem}
\begin{proof}
   For $k=1$ we obtain from \eqref{eq:hb-k-1} that
   \begin{equation}
        \omega_{1}(r) =   \int_r^\infty \int_s^\infty \frac{\e^{-t}}t\mathrm d t\mathrm d s
   \end{equation}
   such that clearly $\omega_{1}(r)>0$ for all $r>0$ and furthermore for $r\ge 1$
   \begin{equation}
        \omega_{1}(r) \le   \int_r^\infty \int_s^\infty \e^{-t} \mathrm d t\mathrm d s =  \e^{-r}.
   \end{equation}
   The remaining estimates follow by induction. Indeed
    \begin{equation}
        \omega_{k+1}(r) =   \int_r^\infty \int_s^\infty \frac{\e^{-t}}t \omega_{k}(t) \mathrm d t\mathrm d s > 0
   \end{equation}
   and
     \begin{equation}
     \begin{split}
        \omega_{k+1}(r) &\le     \int_r^\infty \int_s^\infty {\e^{-t}} \omega_{k}(t) \mathrm d t\mathrm d s\\
       & \le \frac{1}{(k!)^2}  \int_r^\infty \int_s^\infty {\e^{-(k+1)t}} \mathrm d t\mathrm d s\\
       & = \frac{\e^{-(k+1)r}}{((k+1)!)^2}.
        \end{split}
   \end{equation}
   The estimate for the derivative follows analogously (doing one integration less).
\end{proof}
If we can show that for large enough $k$ the sequences $\omega_{k}(1)$ and
$-\omega_{k}'(1)$ are monotonically decreasing, then the Leibniz argument gives again sharp error estimates (always by the next summand).
The integrals defining $\omega_{k}$ can be evaluated in terms of special functions. We will make use of the exponential integral
\begin{equation}
     \mathrm E_\nu(r) = \int_1^\infty \frac{\e^{-rs}}{s^\nu}\mathrm d s, 
\end{equation}
defined in this way for all $\nu\in\C$ and $\Re r>0$. It satisfies
\begin{equation}
     \mathrm E_1(r) = \int_r^\infty \frac{\e^{-s}}s\mathrm d s, \qquad
    \mathrm E_{\nu+1}'(r) = - \mathrm E_\nu(r).
\end{equation}
For $k=1$ we thus obtain
\begin{equation}
\begin{split}
&  \omega_{1}'(r) = -  \mathrm E_1(r),\\
&   \omega_{1} (r)  =  \mathrm E_2(r)=\e^{-r} - r \mathrm E_1(r).
\end{split}
\end{equation}
For $k=2$ we will make use of the recurrence relation $\nu\mathrm E_{\nu+1}(r) = \e^{-r} - r \mathrm E_\nu(r)$ and obtain
\begin{equation}
\begin{split}
   \omega_{2}'(r) &= -\int_r^\infty \frac{\e^{-s}}s \mathrm E_2(s) \mathrm d s \\&= - \int_r^\infty \left( \frac{\e^{-2s}}s - \e^{-s} \mathrm E_1(s) \right) \mathrm d s 
   \\&= \e^{-r} \mathrm E_1(r)-2 \mathrm E_1(2r)
\end{split}
\end{equation}
based on
\begin{equation}
\begin{split}
  \int_r^\infty \e^{-s} \mathrm E_1(s)\mathrm d s &= \e^{-r}\mathrm E_1(r) - \int_r^\infty \frac{\e^{-2s}}s \mathrm d s\\
  &= \e^{-r} \mathrm E_1(r) - \mathrm E_1(2r).
\end{split}
\end{equation}
Integrating again yields
\begin{equation}
\begin{split}
   \omega_2(r) &= \int_r^\infty (2\mathrm E_1(2s) - \e^{-s}\mathrm E_1(s)) \mathrm d s
   \\&= \mathrm E_2(2r) +  \mathrm E_1(2r) - \e^{-r}\mathrm E_1(r) 
   \\&= \e^{-2r} - \e^{-r}\mathrm E_1(r) + (1-2r)  \mathrm E_1(2r).
 \end{split}
\end{equation}
For $k\ge 3$ the functions $\omega_k(r)$ are not expressible in terms of (known) special functions. For the following statement we formally set $\omega_0(r)=1$.

\begin{lem}
Assume that for a number $k_1\ge1$ depending on $\varkappa$ the estimate $\varkappa\omega_{k_1}(r)<\omega_{k_1-1}(r)$ is valid for all $r\ge1$. Then
for all $k\ge k_1(\varkappa)$ the estimates
\begin{equation}\label{eq:om-mon}
\begin{split}
   \varkappa \omega_{k+1}(r) < \omega_k(r),\\  \omega_k'(r)<\varkappa \omega_{k+1}'(r)<0
\end{split}
\end{equation}
are true.
\end{lem}
\begin{proof}
An  induction argument implies
\begin{equation}
\begin{split}
    \varkappa \omega_{k+1}(r) &=  \varkappa \int_r^\infty \int_s^\infty \frac{\e^{-t}}{t} \omega_k(t) \mathrm d t\mathrm d s\\& <  \int_r^\infty \int_s^\infty \frac{\e^{-t}}{t} \omega_{k-1}(t) \mathrm d t\mathrm d s = \omega_k(r)
    \end{split}
\end{equation}
from $\varkappa \omega_k(r)<\omega_{k-1}(r)$ and the statement is proven. The second estimate follows by just doing one integration.
\end{proof}

\begin{cor}
For $\varkappa\in(0,6.7)$ the estimate \eqref{eq:om-mon} holds true for all $k\ge 0$ and $r\ge 1$.
\end{cor}
\begin{proof}
Since $\mathrm E_2(r)$ is decreasing in $r$ it suffices to check that $\mathrm E_2(1) < 1/6.7$, which is true.
\end{proof}

\begin{proof}[Proof of Theorem~\ref{thm3}]
The first zero resonance appears at the first positive zero of the Wronskian
\begin{equation}
\mathcal W(\varkappa)
= \mathbf u_{\rm ext}(1;\varkappa) \mathbf u_{\rm int}'(1;\varkappa) -
\mathbf  u_{\rm int}(1;\varkappa) \mathbf u_{\rm ext}'(1;\varkappa)
\end{equation}
and the proof follows by estimating this function from below and from above. For this we collect some useful numbers,
\begin{equation}
\begin{split}
  \omega_{1}(1) =  \mathrm E_2(1)  \approx  0.148496..., \\
    \omega_{1}'(1) = -\mathrm E_1(1) \approx -0.219384...
\end{split}
\end{equation}
and
\begin{equation}
\begin{split}
    \omega_2(1) &=  \mathrm E_2(2) +  \mathrm E_1(2) - \e^{-1}\mathrm E_1(1) \approx 0.00572793... ,\\ \omega_2'(1) &= \e^{-1} \mathrm E_1(1)-2 \mathrm E_1(2)\approx -0.0170942...
 \end{split}
\end{equation}
Using the monotonicity of $\varkappa^k\omega_k(1)$ and $\varkappa^k \omega_k'(1)$ for $\varkappa\in(0,6.7)$ as well as the monotonicity of $\alpha_k(\varkappa)$, $k\alpha_k(\varkappa)$ shown for $\varkappa\in(0,2.6)$  we get the following rough lower
\begin{equation}
\begin{split}
   \mathcal W(\varkappa) &= \mathbf u_{\rm ext}(1;\varkappa) \mathbf u_{\rm int}'(1;\varkappa) -
\mathbf  u_{\rm int}(1;\varkappa) \mathbf u_{\rm ext}'(1;\varkappa) \\ &>   \bigg(1 - \varkappa \omega_1(1) \bigg) \left(\sum_{k=0}^7 k \alpha_k(\varkappa) \right)\\&\quad -  \bigg( -1 + \varkappa \omega_1(1) - \varkappa \omega'_1(1) \bigg)  \left(\sum_{k=0}^8 \alpha_k(\varkappa) \right)
\end{split}
\end{equation}
and upper bound
\begin{equation}
\begin{split}
   \mathcal W(\varkappa) &= \mathbf u_{\rm ext}(1;\varkappa) \mathbf u_{\rm int}'(1;\varkappa) -
\mathbf  u_{\rm int}(1;\varkappa) \mathbf u_{\rm ext}'(1;\varkappa) \\
<&
   \bigg( 1 - \varkappa \omega_1(1) + \varkappa^2 \omega_2(1)   \bigg) \left(\sum_{k=0}^8 k \alpha_k(\varkappa) \right) \\
& -\bigg( -1 + \varkappa \omega_1(1) - \varkappa^2 \omega_2(1) - \varkappa \omega'_1(1) + \varkappa^2 \omega_2'(1)  \bigg) \\
&\times \left(\sum_{k=0}^7 \alpha_k(\varkappa) \right)
\end{split}
\end{equation}
by polynomials in $\varkappa$. The first resonant value $\varkappa_*$ lies therefore between the first zeros of these polynomials, a numerical calculation gives
the desired estimate $1.67626<\varkappa_*<1.68742$.
\end{proof}
The accuracy of this rough estimate can be increased to any precision by using more terms of the series.


\begin{acknowledgements}
V.~Georgiev was supported in part by  INDAM, GNAMPA - Gruppo Nazionale per l'Analisi Matematica, la Probabilita e le loro Applicazion, by Institute of Mathematics and Informatics, Bulgarian Academy of Sciences and by Top Global University Project, Waseda University.
\end{acknowledgements}

\bibliography{Final-Resonances}
\end{document}